\documentclass[12pt]{amsart}

\usepackage{fullpage}
\usepackage{mathtools}
\usepackage{amsmath,amsfonts,amssymb,amsthm}
\usepackage[mathscr]{eucal}
\usepackage{tikz-cd}
\usepackage{enumitem,kantlipsum}
\usepackage[normalem]{ulem}
\usepackage[numbers]{natbib}
\usepackage{caption}
\usepackage{subcaption}
\usepackage{tikz}
\usepackage{todonotes}

\usetikzlibrary{shapes,positioning,intersections,quotes}

\usepackage{datetime2}

\RequirePackage[colorlinks,citecolor=blue,urlcolor=blue,backref=page]{hyperref}

\usepackage[final]{showlabels} 

\makeatletter
\@namedef{subjclassname@2020}{\textup{2020} Mathematics Subject Classification}
\makeatother

\DeclarePairedDelimiterX\Basics[1](){ #1}

\newcommand{\sP}{\mathbb{P}}  
\newcommand{\sE}{\mathbb {E} } 

\newcommand{\E}{\mathbb {E} }

\newcommand{\BR}{\mathbb{R} }
\newcommand{\cP}{\mathcal{P}}

\newcommand{\1}{\mathbf{1}}
\newcommand{\remove}[1]{}

\newtheorem{theorem}{Theorem}[section]

\newtheorem{lemma}[theorem]{Lemma}
\newtheorem{proposition}[theorem]{Proposition}
\newtheorem{remark}[theorem]{Remark}

\newtheorem{definition}[theorem]{Definition}

\newtheorem{example}[theorem]{ Example}

\numberwithin{equation}{section}

\newcommand{\no}{\nonumber}

\newcommand{\hN}{\hat{N}}

\newcommand{\lk}{\mathrm{lk}}
\newcommand{\tlk}{\tilde{\mathrm{lk}}}
\newcommand{\md}{\mathrm{d}}
\newcommand{\omu}{\overline{\mu}}
\newcommand{\umu}{\underline{\mu}}
\newcommand{\on}{\overline{n}}
\newcommand{\un}{\underline{n}}
\newcommand{\ER}{Erd\"{o}s-R\'enyi }
\newcommand{\tG}{\tilde{G}}
\newcommand{\mL}{\mathcal{L}}
\newcommand{\bU}{\bar{U}}

\begin{document}

\title{Random clique complex process inside the critical window.}

\date{\today}

\author{Agniva Roy} 
\address{School of Mathematics, Georgia Tech, Atlanta, USA. }
\email{aroy86@gatech.edu}

\author{D. Yogeshwaran}
\address{Theoretical Statistics and Mathematics Unit, Indian Statistical Institute, Bangalore.}
\email{d.yogesh@isibang.ac.in}




\begin{abstract}
We consider the  random clique complex process - the process of clique complexes induced by the complete graph with i.i.d. Uniform edge weights.  We investigate the evolution of the Betti numbers of the clique complex process in the critical window and in particular,  show a process-level convergence of the Betti numbers to a Poisson process. Our proof technique gives easily an hitting time result i.e, with high probability, the $k$th cohomology becomes trivial when there are no more isolated $k$-faces. Our results imply that the thresholds for vanishing of cohomology of the clique complex process coincides with that of the threshold for vanishing of `instantaneous' homology determined by \citet{SVT}. We also give a lower bound for the probability of clique complex process to have Kazhdan's property $(T)$.   These results show a different behaviour for the clique complex process compared to the \v{C}ech complex process investigated in the geometric setting by \citet{B19}.  
\end{abstract}

\maketitle

\section{Introduction}
\label{s:intro}

The success of \ER  random graph model \cite{FK} and need for higher-dimensional analogues of random graphs in topological data analysis have driven studies of various models of weighted random (simplicial) complexes \cite{BK21,survey}.  A crucial question here is determining the threshold for homological connectivity in random complexes.  In this article,  we investigate the behaviour of cohomology of the random clique complex process close to the threshold for vanishing of cohomology.  To the best of our knowledge, process level convergence of this model is not straightforward to deduce from existing results due to the lack of monotonicity in homology.  

\paragraph{\bf Organization of the paper.}  In the rest of the introduction,  we introduce the model,  state our main results and also place them in the context of existing literature.   In Section \ref{s:prelims},  we briefly recall some basic notions in combinatorial topology and state simple lemmas for later use.  We give proof of our main theorems in Section \ref{s:proofs}.  To keep our exposition brief,  we shall not describe the background and related literature in detail but provide pointers to the same; we refer the interested reader to the introduction in \citet{B19}.   

\paragraph{\bf Set-up.}  We first define the random graph process.  Let $U(i,j)$,  $1 \leq i < j \leq n$ be independent identically distributed uniform($[0,1]$) random variables.  We set $U(i,j) = U(j,i)$ for $i > j$.  Let $t \in [0,1]$.  Setting $E(n,t) := \{(i,j) : i \neq j,  U(i,j) \leq t \}$,   we define the graph $G(n,t) := ([n], E(n,t))$ where  $[n] := \{1,\ldots,n\}$.  The graph $G(n,t)$ has the same distribution as the well known \ER  random graph with parameters $n,t$.  For convenience,  we shall refer to $G(n,t)$ as the \ER  random graph.  

Associated to a graph $G$,  one can build a (simplicial) complex called the {\em clique complex} (also known as {\em Vietoris-Rips} complex) by considering $k$-cliques as $(k-1)$-faces of the (simplicial) complex.  We refer the reader to Section \ref{s:prelims} for more detailed definitions of various notions used here.   The {\em (\ER) random clique complex proces} $X(n,t), t \geq 0$ is the process of clique complexes associated to the random graph process $G(n,t)$.  

This random clique complex was introduced in \cite{Kahle1} and threshold for homological connectivity was investigated in \cite{Kahle1,SVT,Fowler2019}.   Denoting by $H^k(\cdot)$,  the cohomology group of a complex,  triviality of $H^k(X(n,t))$ is referred to as homological connectivity.  Triviality of $H^0(X(n,t))$ corresponds to graph connectivity of $G(n,t)$ and this has been studied in detail; see for example \cite{FK}.  The crucial difference between $k = 0$ and $k \geq 1$ is monotonicity.  If $H^0(X(n,t))$  is trivial (i.e., $G(n,t)$ is connected) then $H^0(X(n,s))$ is also trivial for all $s \geq t$.   However,  this is not true for $k \geq 1$ and naturally leads to the main question of this article.  Now onwards,  we shall assume that $k \geq 1$.  \\

\paragraph{\bf Main results.}  

The above discussion leads us to consider the random clique complex process $X(n,t),  t \in [0,1]$ and in particular,  the Betti number process $\beta_k(X(n,t)),  t \in [0,1]$ where $\beta_k$ is the rank of the cohomology group $H^k$ with rational coefficients.  Triviality of $H^k$ is equivalent to $\beta_k = 0$.  With this brief background,  our main question is the scaling and asymptotic distribution of the vanishing threshold $T_{n,k}$ when $H^{k}(X(n,t))$ becomes trivial.  More formally,   set
\begin{equation}
\label{d:tnk}
T_{n,k} := \inf \{ t :  \beta_k(X(n,s)) = 0 \,  \forall s \geq t \}.
\end{equation}
To relate $T_{n,k}$ to triviality of $H^k$,  note that 
$$\{T_{n,k} \geq t \} = \{ \beta_k(X(n,s)) \neq 0 \mbox{ for some $s \geq t$} \}.$$
Going further,  we shall investigate the asymptotics of the process $\beta_k(t) := \beta_k(X(n,t)),  t \in [0,1]$ in the critical window (i.e.,  close to its vanishing threshold).  We shall now state our theorem determining the scaling and asymptotic distribution of the process.  
\begin{theorem} 
\label{thm:main}
Let $k \geq 1$.  For $c \in \BR$,   we set $t_c := t_{c}(k,n) = \big(\frac{(\frac{k}{2}+1)\log n + \frac{k}{2} \log \log n +c}{n}\big)^{\frac{1}{k+1}}$.   Then we have that 
$$\{\beta_k(t_c)\}_{c \in \BR} \overset{d}{\to} \{\cP(c)\}_{c \in \BR},$$
where $\{\cP(c)\}_{c \in \BR}$ is a Poisson process on $\BR$ with intensity measure $\mu(k,x) \, \md x$, $\mu(k,x) := \frac{(k/2+1)^{k/2}}{(k+1)!}e^{-x}$ and $\overset{d}{\to}$ denotes weak convergence in the space $D[0,\infty)$ of right-continuous functions with left limits. 
\end{theorem}
The Poisson process $\cP(c)$ is a non-increasing pure jump process defined by two properties (i) $\cP(b) - \cP(a)$ has Poisson($\mu(k,a) - \mu(k,b)$) distribution and (ii) $\cP(a_{2j}) - \cP(a_{2j-1}), j = 1,\ldots,m$ are independent if $-\infty < a_1 < a_2 \ldots < a_{2m} < \infty$.

Asymptotic distribution of $\beta_k(X(n,t_c))$ for a fixed $c$ was shown in \cite[Theorem 1.3]{Fowler2019}. We feel that there is a gap in the arguments therein and our proof of the above theorem does fix the gap; see Remark \ref{rem:gap_approx}. Even so, it is not possible to deduce asymptotic distribution of $T_{n,k}$ from asymptotic distribution of $\beta_k(X(n,t_c))$. But we can easily conclude from Theorem \ref{thm:main} and non-increasing property of the Poisson process that
$$ \lim_{n \to \infty} \sP(T_{n,k} \leq t_c) =  \lim_{n \to \infty}  \sP\big(\bigcap_{t \geq t_c}(\beta_k(t) = 0) \big) = \sP\big(\cP(c) = 0 \big) = e^{-\int_{c}^{\infty} \mu(k,x) \md x} = e^{-\mu(k,c)},$$
i.e.,  $nT_{n,k}^{k+1} - (\frac{k}{2}+1)\log n - \frac{k}{2} \log \log n$ converges in distribution to a Gumbel distribution with parameter $\mu(k,c)$.  This also immediately implies that there are no 'exceptional times'  for triviality of $H^k(X(n,t))$ i.e.,  for $k\geq 1,$ and a sequence $w(n)$ such that $w(n) \to \infty$, it holds that
\begin{equation}
\label{e:vancohom}
 \sP(\bigcap_{t \geq t_n}H^{k}(X(n,t))=0) \to 1,  \,  \,  \mbox{for} \,  \,  t_n :=  \bigg(\frac{(\frac{k}{2}+1)\log n + \frac{k}{2} \log \log n + w(n)}{n}\bigg)^{\frac{1}{k+1}}.
 \end{equation}
As an immediate corollary, we recover \cite[Theorem 1.1]{SVT} - for $k,  w(n),  t_n$ as above,  we have that
 $$ \sP(H^{k}(X(n,t))=0) \to 1,  \,  \,  \mbox{if} \,  \,  t \geq t_n.$$
Thus the thresholds for vanishing of `instantaneous cohomology' and vanishing of cohomology coincide. This is not the case for the random geometric \v{C}ech Complex (see \cite[Theorem 3.1 and Corollary 7.14]{B19}).   Also,  part of our proof involves showing that isolated faces also exhibit similar threshold behaviour  (Proposition \ref{prop:van_isol_faces}) and such a phenomenon has been shown to hold true for $1$-faces in geometric clique complexes as well  \cite[Proposition 1.4 and 1.5]{IY20} but fails for $1$-faces in random \v{C}ech complexes \cite[Proposition 1.4 and 1.5]{IY20}.   

On the other hand like with connectivity in \ER random graph case,  the thresholds for vanishing of `instantaneous cohomology' and vanishing of cohomology coincide trivially for the random $k$-complex (see \cite{LM,MW}) due to monotonicity.  Process-level convergence inside the critical window was shown for this model in \cite[Theorem 7]{STY20} building upon the marginal distribution convergence proven in \cite[Theorem 1.10]{critwind}.   

Another easy consequence of our proof technique is that apart from process level convergence, we also obtain a hitting time result. Even for random $k$-complexes such a result is known only in the case of $k = 1,2$ which were proven respectively in \cite[Theorem 4]{Bollobas1985} and \cite[Theorem 1.11]{critwind}. Hitting time results for the random \v{C}ech complex was shown in \cite{B19} (see discussion below (3.4) therein) with the result for $0$th homology of random \v{C}ech complex (or connectivity in the random geometric graph) proven in \cite{Pen97}. We now state such a result for all $k$ for random clique complexes. 

Let $N_k(t) := N_k(X(n,t))$ denote the number of isolated $k$-faces in $X(n,t)$ which is same as the number of maximal $(k+1)$-cliques in $G(n,t)$. That this approximates $\beta_k(t)$ very well for many values of $t$ has been the driving force behind the results in random clique complexes \cite{SVT,Fowler2019} including ours. Also note that $N_k(t)$ is not monotonic. Analogous to vanishing threshold $T_{k,n}$  
for $\beta_k$, we may define $T'_{n,k}$ as the vanishing threshold for $N_k(t)$ i.e.,
\begin{equation}
\label{d:tnkiso}
T'_{n,k} := \inf \{ t :  N_k(X(n,s)) = 0 \,  \forall s \geq t \}.
\end{equation}
Informally, the hitting time result shows that the $k$th cohomology group of the random clique complex process vanishes when the last isolated $k$-face disappears. We state it formally now and after the proof mention a strengthening of the same.
\begin{theorem}
\label{thm:hitting}
Let $T_{n,k},T'_{n,k}$ be the vanishing thresholds of $k$th Betti number and isolated $k$-faces. Then we have that as $n \to \infty$, 
$$ \sP(T_{n,k} = T'_{n,k}) \to 1.$$
\end{theorem}

Our third and final result is less precise but nevertheless gives a lower bound for the probability that the random clique complex process has Kazhdan's property (T).  
\begin{theorem}
\label{thm:main1}
Let $k, c,  t_c(k,n), \mu(k,c)$ be as in the above Theorem.  Then,
$$ \limsup_{n \to \infty}  \sP\big(\bigcup_{t \geq t_c(1,n)} \{\mbox{$\Pi_1(X(n,t))$ does not have property $(T)$}\}\big) \leq 1 - e^{-\mu(1,c)},$$   
where $\Pi_1(X(n,t))$ is the fundamental group of $X(n,t)$. 
\end{theorem}
As with Theorem \ref{thm:main},  we can also deduce an analogue of \eqref{e:vancohom} and vanishing threshold for Kazhdan's property (T) as in \cite[Theorem 1.2]{SVT} from Theorem \ref{thm:main1}. \\

\paragraph{\bf Proof Outline:}  The proof of \cite[Theorem 1.1]{SVT} uses Garland's method (Theorem \ref{thm:garland}) which reduces vanishing of $\beta_k$ to verifying vanishing of isolated $k$-faces and large spectral gap on the links of $(k-1)$-faces.  In \cite{SVT},  the former is verified by counting arguments and Markov's inequality and the latter via the powerful spectral gap result of \cite[Theorem 1.1]{SPEC}. A similar approach is used in \cite[Theorem 1.3]{Fowler2019} to prove Poisson convergence in the critical window. Our proof involves showing analogue of Theorem \ref{thm:main} for isolated $k$-faces (Theorem \ref{thm:isolatedprocess}) and then showing that the process of isolated $k$-faces and Betti numbers coincide starting anywhere in the critical window (Theorem \ref{thm:bknkapprox}).  The latter result is the main technical contribution of this paper and allows to translate results about isolated $k$-faces to $k$th Betti number at a process-level. To prove this result, we use that $\beta_k(X(n,t))$ is a jump-process along with a Borel-Cantelli argument to reduce the proof to obtaining good estimates on the probabilities $\sP(\beta_k(X(n,t)) \neq N_k(X(n,t)))$. These estimates need to satisfy suitable integrability in $t$ and summability in $n$. For the same, we quantify the probability estimates in \cite{SVT,Fowler2019} much more carefully to enable us to derive the desired bounds via Garland's method and the spectral gap bound of \cite{SPEC}. Lastly, the same method with Zuk's criteria (Theorem \ref{thm:zuk}) replacing Garland's method yields Theorem \ref{thm:main1}. We make a more precise comparison after Remark \ref{rem:specgap}. Our proofs would simplify significantly if one is interested in the above results for $t >> t_c$; see Remark \ref{rem:gap_approx}. It should be possible to extend our methods to the multi-parameter random clique complex model as in \cite{Fowler2019,Costa16,Costa17,Costa17a}.   

\section{Preliminaries}
\label{s:prelims}
In this section we recall some basic definitions regarding (simplicial) complexes as well as collect some topological results that are used later.  The reader can turn to the books \citet{EH} and \citet{Mun} for more background on complexes and details. We shall assume familiarity with basic algebraic topology notions.

\begin{definition} A {\em simplicial complex} is a family $\mathcal{S}$ of finite sets such that if $A \in \mathcal{S}$, and $B \subseteq A$, then $B \in \mathcal{S}$. 
\begin{itemize}
\item The vertex set of $\mathcal{S}$ is the union of all its constituent sets. $V(\mathcal{S}) = \cup_{A \in \mathcal{S}} A$.
\item Every set $A \in \mathcal{S}$ is called a $k$-{\em face} of $\mathcal{S}$ where $|A| = k+1$. 
\item $A \in \mathcal{S}$ is called {\em isolated} or {\em maximal} if $\nexists$ $B \in \mathcal{S}$ such that $A \subset B$.
\item The $k$-skeleton of $\mathcal{S}$ is defined to be the sub simplicial complex consisting of all the faces of $\mathcal{S}$ that have size atmost $(k+1)$. The 1-skeleton is also referred to as the {\em underlying graph}.
\end{itemize}
\end{definition}

Every simplicial complex has well-defined topological invariants associated to it, namely its homology and cohomology groups. The reader can refer to \cite{EH} for background on how they are defined. The $k$-th cohomology group of a space $X$ with coefficients in the field $F$ will be denoted $H^k(X, F)$. The rank of $H^k(X, \mathbb{Q})$, called the {\em Betti number}, will be denoted $\beta_k(X)$.

We recall now some very useful criteria to verify homological connectivity of complexes given in terms of spectral gap of certain graphs.
\begin{definition}
The spectral gap of a graph, denoted $\lambda_2$, is defined to be the second largest eigenvalue of the symmetric normalised Laplacian of the graph $L = I - D^{-\frac{1}{2}}AD^{-\frac{1}{2}}$, where $D$ is the degree matrix and $A$ is the adjacency matrix.
\end{definition}
\begin{theorem}(\citet{G2},  \citet[Theorem 2.5]{G1}) 
\label{thm:garland}
Let $X$ be a pure $k+1$-dimensional finite simplicial complex.  If for every $(k-1)$-face $\sigma$,  the link $lk_{X}(\sigma)$ is connected and has spectral gap $\lambda_{2}[lk_{X}(\sigma)] > 1 - \frac{1}{k+1}$, then $H^{k}(X, \mathbb{Q}) = 0$.
\end{theorem}
\begin{theorem}(\citet[Theorem 1]{zuk2003property}) \label{thm:zuk} If $X$ is a pure 2-dimensional locally-finite simplicial complex such that for every vertex $v$, the vertex link $lk(v)$ is connected and the normalized Laplacian $L = L[lk(v)]$ satisfies $\lambda_2(L) > 1/2$, then $\pi_1(X)$ has property (T).

\end{theorem}
%


Now we shall state and prove some lemmas which will help us to understand how removal of isolated faces affects Betti numbers.  Though these are implicitly used in \cite[Theorem 1.2]{Fowler2019},  we explicitly state and prove them here for convenience as well as to delineate the deterministic parts of our proofs from the probabilistic parts.

Given a simplicial complex $X$,  define $V_{k} \subset V $,  as the set of vertices which are isolated vertices in the link of some $(k-1)$-face of $X$.  Define $\Sigma$ as the set of maximal $k$-faces. Suppose $\Sigma = \{\sigma_1, \dots, \sigma_m\}$.  Define $X' = X \setminus \Sigma$,  the simplicial complex obtained from $X$ by deleting all the maximal $k$-faces. 

\begin{lemma}\label{l:lkX'conn}
A vertex $v \in V_{k}$ if and only if $v$ is the vertex of some $k$-face in $\Sigma$. It follows that 
\begin{enumerate}
\item $V_{k}$ is empty iff $\Sigma$ is empty.
\item Let $\tau$ be a $(k-1)$-face of $X'$.  Then $lk_{X'}(\tau)$ is connected if and only if $lk_{X}(\tau)$ is a graph with one non-trivial component (i.e.,  a component with at least two-vertices) and the rest are isolated vertices.
\end{enumerate}
\end{lemma}

\begin{proof}
We prove (1) first. Suppose $v \in V_k$. Then, there exists a $k$-face $\sigma$ with vertices $a_1, \dots, a_{k+1}$ in $X$ such that $v \in lk_X(\sigma)$, i.e., $a_1, \dots, a_{k+1}, v$ is a $(k+1)$-face of $X$.  Further, $v$ is isolated in $ lk_X(\sigma)$, which means that there is no vertex $y$ such that $a_1, \dots, a_{k+1}, v, y$ are the vertices a $(k+2)$-face of $X$. This means that $\sigma$ is a maximal $(k+1)$-face, i.e., $\sigma \in \Sigma$.

Conversely, suppose $v$ is the vertex of $\sigma \in \Sigma$. Suppose the vertices of $\sigma$ are $a_1, \dots, a_{k+1}, v$. Call the $k$-face with vertices $a_1, \dots, a_{k+1}$ as $\tau$. Then, $v \in lk_X(\tau)$. Also, since there is no $(k+2)$-face such that $\sigma$ is a subface of that, $v$ is an isolated vertex in $lk_X(\tau)$, which means $v \in V_k$. 

(2) follows clearly by observing that $lk_{X'}(\tau)$ is obtained by removing vertices corresponding to maximal $k$-faces containing $\tau$ from $lk_{X}(\tau)$.
\end{proof}

\begin{lemma}\label{l:bettikX'X}
Let $X$ be a simplicial complex as above with exactly $m$ maximal faces. Suppose $X'$ is obtained from $X$ by removing all the maximal $k$-faces of $X$. Then, $\beta_k(X) > m$ implies $\beta_k(X') \geq 1$.
\end{lemma}

\begin{proof}
Suppose $\beta_k(X) > m$. Thus, $H^k(X)$ is an abelian group with at least $(m+1)$ infinite order generators. Consider a generating set of $H^k(X)$ that contains $[\phi_1],\dots,[\phi_m]$, the cohomology classes corresponding to the characteristic functions of the $m$ maximal $k$-faces $\sigma_1, \dots, \sigma_m$. As the $\sigma_i$ are maximal, it is clear that $\delta^k (\phi_i) = 0$, so they are well-defined cohomology classes. As the $[\phi_i]$ may or may not be linearly independent, the generating set for $H^k(X)$ would have at least one other infinite order generator. The group $H^k(X')$ is obtained from $H^k(X)$ with the same generating set, but by adding the relations $[\phi_1]=1,\dots,[\phi_m]=1$. However, as there is at least one more generator not involving any of the $[\phi_i]$, $H^k(X')$ has at least one infinite order generator. Thus $\beta_k(X') \geq 1$.
\end{proof}

\begin{lemma}\label{l:xtox'}
Let $X$ be a simplicial complex as above with exactly $m$ maximal faces. Suppose $X'$ is obtained from $X$ by removing all the maximal $k$-faces of $X$. Then, $\beta_k(X) < m$ implies $\beta_{k-1}(X') \geq 1$.
\end{lemma}

\begin{proof}
Call the characteristic functions of the $m$ maximal faces $\phi_1, \dots, \phi_m$. These are $k$-cocycles, as they correspond to maximal faces. Then, if $\beta_k (X) < m$, the $\phi_i$ must be linearly dependent in $H^{k}(X)$, i.e., there must be a $(k-1)$ cochain $\lambda$ such that $\delta^{k-1}(\lambda) = \sum_{i=1}^m a_i\phi_i$. Thus, it follows that $\lambda$ is not a $(k-1)$-coboundary. Now consider $X' = X - \{\sigma_1, \dots , \sigma_{m}\}$. Now, $\delta^{k-1}(\lambda)=0$ in $X'$, and $\lambda$ is not a coboundary in $X$ or $X'$, which means $\lambda$ would give a nontrivial element in $H^{k-1}(X')$. 
\end{proof}

\section{Proofs}
\label{s:proofs}

We give proofs of our main theorems in this section.  The analogue of Theorem \ref{thm:main} for isolated faces is proved in Section \ref{s:vanisolfaces}.  We relate the process of isolated faces to Betti number process in Section \ref{s:frombknk} with proof of some technical lemmas postponed to Section \ref{s:problemmas}.  In Section \ref{s:proofthmmain1},  we complete the proof of our main theorems. 

\subsection{Poisson process convergence of Isolated faces - Theorem \ref{thm:isolatedprocess}}
\label{s:vanisolfaces}

Recall that $N_k(t)$ is the number of maximal $(k+1)$-cliques in $G(n,t)$ or the number of isolated $k$-faces in $X(n,t)$.  We can define this formally as
$$ N_k(t) :=  \frac{1}{(k+1)!} \sum^{\neq}_{i_1,\ldots,i_{k+1}} \1[ \mbox{$i_1,\ldots,i_{k+1}$ form a maximal clique in $G(n,t)$. }],$$
where  $\sum^{\neq}$ denotes that the sum is over distinct indices.  Since we are interested in persistent homology,  we need to understand maximal cliques that exist for any $s > t$.  We define these as 
$$ N_k^*(t) :=  \frac{1}{(k+1)!} \sum^{\neq}_{i_1,\ldots,i_{k+1}} \1[ \mbox{$i_1,\ldots,i_{k+1}$ form a maximal clique in $G(n,s)$ for some $s \geq t$. }].$$
Note that while $N_k(t)$ is non-monotonic in $t$,  $N_k^*(t)$ is monotonic.  Trivially,  we have that $N_k(t) \leq N_k^*(t)$ but we show that both are asymptotically equal at $t_c$ onwards.  
\begin{proposition}
\label{prop:van_isol_faces}Let $k \geq 1$ and $t_c(k,n) =  \big(\frac{(\frac{k}{2}+1)\log n + \frac{k}{2}\log \log n + c}{n}\big)^{\frac{1}{k+1}}$ for $n \geq 3$.   Then we have that as $n \to \infty$, 
\begin{equation}
\label{e:cvgnknk*}
 \E[N_k^*(t_c(k,n)) - N_k(t_c(k,n))] \to 0.
 \end{equation}
Consequently,  $N_{k}^{*}(t_c(k,n))$ converges in distribution to a Poisson random variable with mean $\mu(k,c)$ as in Theorem \ref{thm:main}.
\end{proposition}

 \begin{proof}
We shall only prove \eqref{e:cvgnknk*} as the Poisson convergence follows from \eqref{e:cvgnknk*},  Poisson convergence of $N_k(t_c(k,n))$ \cite[Theorem 2.3]{SVT} (see also \cite[Lemma 7.1]{Fowler2019}) and Slutsky's theorem.  

From now on,  we abbreviate $t_c(k,n)$ by $t_c$ and drop the subscript $k$ for convenience.  Let $t \geq t_c$.  Further set $\hN(t) := N^*(t) - N(t)$.   Thus to prove \eqref{e:cvgnknk*},  it suffices to show that
\begin{equation}
\label{e:cvghnk}
\lim_{n \to \infty} \E[\hN(t_c)] = 0.
\end{equation} 
We can represent $\hN(t)$ as
\begin{align*}
 \hN(t) &:=  \frac{1}{(k+1)!} \sum^{\neq}_{i_1,\ldots,i_{k+1}} \1[ \mbox{$i_1,\ldots,i_{k+1}$ do not form a clique in $G(n,t)$}]  \\
 & \quad \quad \times \1[ \mbox{$i_1,\ldots,i_{k+1}$ form a maximal clique in $G(n,s)$ for some $s > t$}].
 \end{align*}
By the construction of $G(n,t)$ via the weights $U(i,j)$,  we re-write the above events.  We will do so assuming $i_1 = 1,\ldots,i_{k+1} = k+1$ for notational convenience. 
\begin{align*}
\{\mbox{$1,  \ldots , k+1$ do not form a clique in $G(n,t)$}\} & = \{ \max_{l, m = 1,\ldots,k+1} U(l,m) > t \},  \\
\{  \mbox{$1,  \ldots , k+1$ form a maximal clique in $G(n,s)$} \} &= \{  \max_{l, m = 1,\ldots,k+1} U(l,m) \leq s \} \\
& \quad \quad \cap \bigcap_{j > k+1} \{\max_{1 \leq l \leq k+1} U(j,l) > s\} .
\end{align*}
From the above identities,  note that if $1,\ldots,k+1$ form a maximal clique in $G(n,s)$ for some $s$,   then necessarily they form a maximal clique in $G(n,s)$ for $s' = \max_{l, m = 1,\ldots,k+1} U(l,m)$.  Using this observation and combining the above events, 
we have that 
\begin{align*}
& \{ \mbox{$1,  \ldots , k+1$ do not form a clique in $G(n,t)$ but form a maximal clique in $G(n,s)$ for some $s > t$}\} \\
& = \bigcup_{1 \leq i \neq j \leq k+1} \left(  \{ U(i,j) > t,  \max_{l, m = 1,\ldots,k+1} U(l,m) \leq U(i,j) \} \bigcap \bigcap_{p > k+1} \{ \max_{l =1, \ldots,  k+1} U(p,l) > U(i,j) \} \right) \\
& =:  \bigcup_{1 \leq i \neq j \leq k+1}A(i,j).  
\end{align*}
We have from the definition of $\hN(t)$,  the above identity,  union bound and exchangeability that 
\begin{equation}
\label{e:unionbdhnt}
\E[\hN(t)]  \leq \binom{n}{k+1}\binom{k+1}{2} \sP(A(1,2)).
\end{equation}
By conditioning on $U(1,2)$ and then using the independence of $U(\cdot,\cdot)$'s,  we have that
$$\sP(A(1,2)) = \int_{t}^{1} s^{\binom{k+1}{2} - 1}(1 - s^{k+1})^{(n-k-1)}ds.$$
Now by substituting the above probability in \eqref{e:unionbdhnt} and using some elementary bounds,  we obtain that for a constant $C$ depending on $k$,  
$$ \E[\hN(t)] \leq Cn^{k+1}\int_{t}^{1}s^{\binom{k+1}{2} - 1}e^{-ns^{k+1}}\md s.$$
Now using that  $s^{\binom{k+1}{2}} = s^{k(k+1)/2}$ is increasing in $s$ and $s^{-1}e^{-ns^{k+1}}$ is decreasing in $s$,   we can derive the following bounds. Let $c_n \to \infty$ such that $c_n = o(\log \log n)$ and set $t'_n := (\frac{(k+2)\log n}{n})^{1/(k+1)}$. 
\begin{align*}
& \E[\hN(t_c)]  \leq C  \int_{t_c}^{t_{c_n}} n^{\frac{k}{2}+1}(ns^{k+1})^{k/2} s^{-1}e^{-ns^{k+1}} \md s +  C \int_{t_{c_n}}^{t'_n} n^{\frac{k}{2}+1}(ns^{k+1})^{k/2} s^{-1}e^{-ns^{k+1}} \md s  \\
 & \quad  +  C \int_{t'_n}^1 n^{k+1} e^{-ns^{k+1}}\md s  \\
& \leq C n^{\frac{k}{2}+1}(nt_{c_n}^{k+1})^{k/2}e^{-nt_c^{k+1}}  \int_{t_c}^{t_{c_n}}s^{-1} \md s +  C n^{\frac{k}{2}+1}((k+2)\log n)^{\frac{k}{2}}e^{-nt_{c_n}^{k+1}}  \int_{t_{c_n}}^{t'_n}  s^{-1} \md s + Cn^{-1} \\
& \leq C e^{-c}\bigg( \frac{k}{2}+1 + \frac{k}{2}\frac{\log \log n}{\log n} + \frac{c_n}{\log n}\bigg)^{k/2} \log(\frac{t_{c_n}}{t_c}) +  C e^{-c_n} (k+2)^{k/2} \log(\frac{t'_n}{t_{c_n}}) + Cn^{-1} \\
& \leq C_{k,c}\bigg( \log(\frac{t_{c_n}}{t_c})  + e^{-c_n}  \log(\frac{t'_n}{t_{c_n}})  \bigg) + Cn^{-1}. 
\end{align*}
Now,  as $n \to \infty$,  $t_{c_n}/t_c \to 1$,  $c_n \to \infty$,  $\frac{t'_n}{t_{c_n}} \to 2^{1/(k+1)}$ and so $\E[\hN(t_c)] \to 0.$
\end{proof}

\begin{theorem} 
\label{thm:isolatedprocess}
Let $k \geq 1$.  For $c \in \BR$,   we set $t_c := t_c(k,n) = \big(\frac{(\frac{k}{2}+1)\log n + \frac{k}{2} \log \log n +c}{n}\big)^{\frac{1}{k+1}}$.   Then we have that 
$$\{N_k(t_c)\}_{c \in \BR} \overset{d}{\to} \{\cP(c)\}_{c \in \BR},$$
where $\{\cP(c)\}_{c \in \BR}$  is a Poisson process on $\BR$ with intensity measure $\mu(k,x) \, \md x$ and convergence $\overset{d}{\to}$ as in Theorem \ref{thm:main}. 
\end{theorem}
\begin{proof}
Firstly, we explain that it suffices to prove convergence of the associated point process for random counting measures and then reduce it to Poisson convergence for suitable random variables. Lastly, we shall establish Poisson convergence by a factorial moment argument. We shall outline the main steps but will exclude some derivations and instead point the reader to suitable references for similar calculations. The first step above is standard in weak convergence of pure jump processes, the second step will be based on well-known criteria for convergence of point processes and approximation as in Proposition \ref{prop:van_isol_faces}. The last step involves computing factorial moments similar to those in \cite[Proposition 33]{STY20} and \cite[Lemma 7.1]{Fowler2019}. Here we shall only mention the key points of the computation.

Let $M_p(\BR)$ be the space of locally-finite (Radon) counting measures on $\BR$ equipped with the topology of vague convergence; see \cite[Chapter 3]{Res13} for details. Since $N_k(t_c)$ is a pure jump process, we can represent $N_k(\cdot), c \in \BR$ as an element of $M_p(\BR)$ by considering the time of jumps i.e., for a Borel measurable set $A \subset \BR$, 
\[ N_k(A) = \sum_{c \in A} \1[N_k(t_c) \neq N_k(t_c-)], \]
and similarly $\cP(\cdot)$ can also be represented as a random Radon counting measure. Due to continuous mapping theorem (see the proof of Part I of Theorem 2.2 in \cite{owada2018limit}), it suffices to prove weak convergence of $N_k$ in $M_p(\BR)$ i.e.,
\[ N_k \Rightarrow \cP \, \, \mbox{in} \, \, M_p(\BR).\]

From the above representation, $N_k((c,\infty))$ counts the number of isolated faces at $t_c$ and the isolated faces that are formed after $t_c$ are counted twice - once when they are created and once when they become non-isolated.  Thus we have that
\[ N_k((c,\infty)) = N_k(t_c) + 2\hat{N}_k(t_c) .\]
Since \cite[Lemma 7.1]{Fowler2019} gives that $\E[N_k(t_c)] \to \E[\cP(c)]$ for all $c \in \BR$, we can immediately obtain from Proposition \ref{prop:van_isol_faces} that $\E[N_k(I)] \to \E[\cP(I)]$ for all $I$ which are finite-unions of disjoint intervals.  Since the intensity measure of $\cP$ is non-atomic, to complete the proof of weak convergence in $M_p(\BR)$ it suffices to show that
\begin{equation}
\label{e:voidprobcvg}
\sP(N_k(I) = 0) \to \sP(\cP(I) = 0) = e^{-\mu(k,I)} \, \, \mbox{as} \, \, n \to \infty,
\end{equation}
for all $I$ which are finite-unions of disjoint intervals and $\mu(k,I) := \int_I \mu(k,x) \, \md x$ ; see \cite[Proposition 3.22]{Res13}. 

Let $I = \cup_{j=1}^l (a_{2j-1},a_{2j}]$ for $-\infty< a_1 < \ldots < a_{2l} < \infty$.  Fix $c \in (-\infty, a_1)$. Define $\bar{u} = nu^{k+1} - (\frac{k}{2}+1) \log n - \frac{k}{2} \log \log n$ for $u \in [0,1]$ and $n \geq 3$. We omit $n$ from argument of $\bar{u}$ for convenience.  Define
\begin{align*}
N'(I) & :=  \frac{1}{(k+1)!} \sum^{\neq}_{i_1,\ldots,i_{k+1}} \1[ \mbox{$i_1,\ldots,i_{k+1}$ form a maximal clique in $G(n,t_c)$} \\ 
& \quad \times \1[\min_{j \neq i_1,\ldots, i_k+1} \max \{\bU(j,i_1), \ldots, \bU(j,i_{k+1}) \} \in I], 
\end{align*}
which is nothing but the number of maximal cliques in $G(n,t_c)$ that vanish in $I$.  Setting $\bU(j,\sigma) = \max \{ \bU(j,i) : i \in \sigma \}$, we re-write $N'(I)$ as
$$ N'(I) =  \sum_{\sigma \in \mathcal{K}_{k-1}(n)} \1[ \mbox{$\sigma$ is a clique in $G(n,t_c)$}] \1[\min_{j \notin \sigma} \bU(j,\sigma) \in I],$$
where we are using that if $\min_{j \notin \sigma} \bU(j,\sigma) \in I$ then $\sigma$ has to be maximal in $G(n,t_c)$. Since $\hat{N}(t_c)$ is asymptotically $0$ by Proposition \ref{prop:van_isol_faces}, there are no new maximal cliques created and hence asymptotically the only change in the process $N_k$ is due to vanishing of maximal cliques in $G(n,t_c)$. Thus \eqref{e:voidprobcvg} follows if we show that
 $$N'(I) \overset{d}{\to} \cP(I) \, \, \mbox{as} \, \, n \to \infty,$$
where $\overset{d}{\to}$ is standard convergence in distribution for random variables.  The same can be done via a factorial moment computation \cite[Theorems 2.4 and 2.5]{vdH2017}. For $I = (c,\infty)$, this follows from \cite[Lemma 7.1]{Fowler2019} and Proposition \ref{prop:van_isol_faces}. We only sketch the details of the generic case here.

For notational convenience, we represent $(k+1)$-tuples or $k$-faces by $\sigma$ and we shall fix the standard total ordering on $[n]$ for convenience. Also set $k^{(r)} = k(k-1)\ldots(k-r+1)$ as the $r$th factorial power of $k$ for $k > r$. Under such a notation, the $r$th factorial of $N'(I)$ can be written as 
$$N'(I)^{(r)} = \sum^{\neq}_{\sigma_1,\ldots,\sigma_r} \prod_{i=1}^r \1[ \mbox{$\sigma_i$ is a clique in $G(n,t_c)$}] \1[\min_{j \notin \sigma_i} \bU(j,\sigma_i) \in I], $$ 
and so its $r$th factorial moment is
$$ \E[N'(I)^{(r)}] =  \sum^{\neq}_{\sigma_1,\ldots,\sigma_r} \sP\Big( \{ \mbox{$\sigma_i$'s are cliques in $G(n,t_c)$} \} \cap \{ \min_{j \notin \sigma_i} \bU(j,\sigma_i) \in I \} \Big),$$
We need to show that as $n \to \infty$, $\E[N'(I)^{(r)}] \to  \E[\cP(I)^{(r)}] = \mu(k,I)^r$.  This is similar to the proofs of \cite[Proposition 33]{STY20} and \cite[Lemma 7.1]{Fowler2019}. Suppose we abbreviate $ \1[ \mbox{$\sigma_i$ is a clique in $G(n,t_c)$}] \1[\min_{j \notin \sigma_i} \bU(j,\sigma_i) \in I]$ by $\1[\sigma;I]$ for any set $I$, then observe that
$$\prod_{j=1}^r\1[\sigma_j; I] =  \sum_{\alpha_1,\ldots,\alpha_r \in \{1,\ldots,2l\}^r } \prod_{j=1}^r(-1)^{\alpha_j+1} \1[\sigma_j;(a_{\alpha_j},\infty)].$$
We use $\sum^*_{\sigma_1,\ldots,\sigma_r}$ to denote that the $\sigma_i$'s are disjoint i.e., have no common vertices. In this case, we can derive using independence of the indicators inside the sum that
$$\sum^*_{\sigma_1,\ldots,\sigma_r} \sE[\prod_{j=1}^r\1[\sigma_j;(a_{\alpha_j},\infty)]] = \prod_{j=1}^{r} \binom{n - (j-1)(k+1)}{k+1}t_c^{\binom{k+1}{2}}(1 - \sP(\bU \leq a_{\alpha_j})^{k+1})^{n-k-1} \big(1 + o(1) \big),$$
where $U$ is a Uniform $[0,1]$ random variable and the $o(1)$ term is due to the fact that the edge-weights between the vertices of $\sigma_i$ also satisfy the requisite condition. From this one can derive that
$$ \lim_{n \to \infty} \sum^*_{\sigma_1,\ldots,\sigma_r} \sE[\prod_{j=1}^r\1[\sigma_j;(a_{\alpha_j},\infty)]] = \prod_{j=1}^r\mu(k,(\alpha_j,\infty)),$$
and now summing over $\alpha_i's$, we derive that

$$\lim_{n \to \infty} \sum^*_{\sigma_1,\ldots,\sigma_r} \sE[\prod_{j=1}^r\1[\sigma_j;I]] = \mu(k,I)^r.$$
To complete the proof, one needs to show that summation over $\sigma_i$'s that are not disjoint do not contribute asymptotically. We skip this part as it is similar to that in the proof of \cite[Lemma 7.1]{Fowler2019}.

\end{proof}

\subsection{From isolated faces to Betti numbers - Theorem \ref{thm:bknkapprox}.}
\label{s:frombknk}

The bulk of our work is in relating the process of isolated nodes to that of Betti numbers and this is our main theorem that helps us to do so.   
\begin{theorem}
\label{thm:bknkapprox} Let $k \geq 1$ and $t_c :=t_c(k,n) =  \big(\frac{(\frac{k}{2}+1)\log n + \frac{k}{2}\log \log n + c}{n}\big)^{\frac{1}{k+1}}$ for $n \geq 3,  c \in \BR$.   Then we have that as $n \to \infty$,  
\begin{equation}
\label{e:bknkallt}
\lim_{n \to \infty} \sP\Big( \bigcup_{t_c \leq t \leq 1} \{ \beta_k(X(n,t)) \neq N_k(X(n,t)) \} \Big) \to 0.
\end{equation}
%
\end{theorem}
From Proposition \ref{prop:van_isol_faces},  we know that the probability that $X(n,t)$ is a pure $(k+1)$-dimensional complex for all $t \geq t_c$ i.e.,  $X(n,t)$ has finitely many (even though random) isolated $k$-faces. So thanks to Lemmas \ref{l:bettikX'X} and \ref{l:xtox'}, to prove Theorem~\ref{thm:main} we need to show triviality of random complexes obtained by removing these finitely many isolated faces. After removing these isolated faces, the random complex is a pure $(k+1)$-dimensional complex and we shall aim to show triviality of these random complexes via the Garland's method (Theorem \ref{thm:garland}). This requires verification of a certain spectral gap condition. An important ingredient for such a verification is the following spectral gap result for \ER random graphs.
\begin{theorem} (\cite[Theorem 1.1]{SPEC}) 
\label{thm:specgapER}
Fix $\delta > 0$ and let $p > \frac{(\frac{1}{2} + \delta)\log n}{n}$.  Let $d = p(n-1)$ denote the expected degree of a vertex. For every $\epsilon >0$, there is a constant $C = C(\delta,\epsilon)$, so that $\sP (|\lambda_2(\tilde{G}(n,p)) - 1|> \frac{C}{\sqrt{d}}) \leq Cn \exp (-(2-\epsilon)d) + C \exp (-d^{\frac{1}{4}} \log n)$, where $\tilde{G}$ represents the giant component of the graph $G$ i.e.,  the largest component.  
\end{theorem}
The largest component is well-defined for $G(n,p)$ with $p$ as above; for example,  see \cite[Lemma 5.8]{SPEC}.
\begin{remark}
\label{rem:specgap}
The above bound can be interpreted as saying that if $p > \frac{(\alpha + 1)\log n}{n}$,  then for some $C = C(\alpha)$,   we have that 
$$\sP (|\lambda_2(\tilde{G}(n,p)) - 1|> \frac{C}{\sqrt{d}}) \leq Cn^{-2\alpha - 1/2} .$$
 This is because $n \exp (-(2-\epsilon)d)$ dominates $\exp (-d^{\frac{1}{4}} \log n)$, and $n \exp (-(2-\epsilon)d) < n \exp (-(2-\epsilon)(\alpha + 1)\log n) = n^{1 - (2-\epsilon)(1+\alpha)} = n^{-2\alpha - 1 + \epsilon(1+\alpha)}$.  Choosing $\epsilon = 1/2(1+\alpha)$ yields the inequality above.  
\end{remark}
In \cite{SVT} or \cite{Fowler2019}, a different version of Theorem 3.4 is used where the probbability estimate for $G(n,p)$ to be connected and have a large spectral gap is atmost $Cn^{-\alpha}$. Such a bound doesn't suffice for us and instead we use the above bound and estimate separately the probability of an \ER \, random graph consisting of a 'giant component' and isolated vertices. This decays sufficient fast close to connective regime; see STEPS 2,3 and 4 of the proof. We remark again on the proof simplifications for $t >> t_c$ and the neccessity for remaining steps in Remark \ref{rem:gap_approx}.

We first state a technical lemma necessary for proof of Theorem \ref{thm:bknkapprox} but shall defer its proof to the end.  For $m \geq 1$,  set
$$q(\alpha, m)  = \frac{(\alpha + 1)\log m}{m}.$$
\begin{lemma}
\label{lem:qalpham}
Let $\alpha = \frac{k(k+3)}{2} - \rho$ for $\rho > \rho_0$ for a $\rho_0$ small.  Set $\mu(t) := (n-k)t^k$ and $\umu(t) := \mu(t) - (\mu(t))^{3/5}$.  Then $\exists n_0 \in \mathbb{N}$ such that $\forall n \geq n_0$, $m \geq 1$ and $t \geq t_c$, we have that if $t < q(\alpha,m)$ then
$m < \umu(t)$.  
\end{lemma}
\begin{proof}(Proof of Theorem \ref{thm:bknkapprox})
Recall that we set $\beta_k(t) = \beta_k(X(n,t))$ for notational convenience.  We break the proof into six steps to follow it easily.  In the first step,  we reduce the proof to deriving quantitative bounds for certain probabilities at a fixed time $t$.  At the end of \textsc{STEP 1},  we explain what is done in each of the remaining steps.  We shall be using some topological lemmas from Section \ref{s:prelims} and probabilistic estimates from Section \ref{s:problemmas}.  \\

\noindent \underline{\textsc{Step 1: Reduction to estimates at a fixed time $t$.}}
Observe that
\begin{align*}
 \sP\Big( \bigcup_{t_c \leq t} \{ \beta_k(t) \neq N_k(t) \} \Big) & \leq  \sP\Big( \bigcup_{t_c \leq t} \{ \beta_k(t) \neq N_k(t) \} \cap  \{N_k(t_c) \leq m \} \Big) + \sP\Big(N_k(t_c) \geq m  \Big).
\end{align*}
Since $N_k(t_c)$ converges in distribution (see Proposition \ref{prop:van_isol_faces}),  we can choose $m$ large to make the second term small and so it suffices to show that for all $m \geq 1$,  
\begin{equation}
\label{e:bknkallt1}
\lim_{n \to \infty}\sP\Big( \bigcup_{t_c \leq t} \{ \beta_k(t) \neq N_k(t) \} \cap  \{N_k(t_c) = m \} \Big) = 0.
\end{equation}
From the proof of Proposition \ref{prop:van_isol_faces},  we know that $\sE[\hat{N}_k(t_c)] \to 0$.   Thus \eqref{e:bknkallt1} follows, if we show that for all $m \geq 1$
\begin{equation}
\label{e:bknkallt2}
\lim_{n \to \infty} \sP\Big( \bigcup_{t_c \leq t} \{ \beta_k(t) \neq N_k(t),  N_k(t_c) = m,  \hat{N}_k(t_c)= 0 \} \Big) = 0.
\end{equation}
We now define $R^*_{k-1,m}(t), R_{k-1,m}(t) $ motivated by similar definitions in the proof of  \cite[Theorem 1.3]{Fowler2019}.  
\begin{align}
\label{e:rk-1m*t} R_{k-1,m}^*(t) & \coloneqq \frac{1}{k!} \sum_{i_1,\dots,i_k}^{\neq} \1[ \mbox{$i_1,\ldots,i_{k}$ form a $(k-1)$-face which}\\ 
& \qquad \mbox{is a subface of $m$ or fewer $k$-faces in $X(s)$ for some $s \geq t$. }] \no \\
\label{e:rk-1mt} R_{k-1,m}(t) & \coloneqq  \frac{1}{k!} \sum_{i_1,\dots,i_k}^{\neq} \1[ \mbox{$i_1,\ldots,i_{k}$ form a $(k-1)$-face which} \\
& \qquad  \mbox{is a subface of $m$ or fewer $k$-faces in $X(t)$ }] \no
\end{align}
We can break up the relevant probability into the following terms:
\begin{align}
&  \sP\Big( \bigcup_{t_c \leq t} \{ \beta_k(t) \neq N_k(t),  N_k(t_c) = m,  \hat{N}_k(t_c)= 0 \} \Big)  \leq  \sP(R^*_{k-1,m}(t_c) > 0) \no \\
 & \quad + \sP\Big( \bigcup_{t_c \leq t } \{ \beta_k(t) > N_k(t),  N_k(t_c) = m,  \hat{N}_k(t_c)= 0,   R^*_{k-1,m}(t_c) = 0\} \Big)\no  \\
\label{e:bknkallt3} & \quad + \sP\Big( \bigcup_{t_c \leq t } \{ \beta_k(t) < N_k(t),  N_k(t_c) = m,  \hat{N}_k(t_c)= 0,   R^*_{k-1,m}(t_c) = 0\} \Big)
\end{align}
From upcoming Lemma \ref{l:rk*tc},  we have that the first term in the inequality \eqref{e:bknkallt3} converges to $0$ as $n \to \infty$.  To complete the proof,  we need to show that the second and third terms in the inequality \eqref{e:bknkallt3} vanish asymptotically. To prove the same, we shall show that 
\begin{equation}
\label{e:sumzn}
\sum_{n \geq 1}\sE[Z_n] < \infty 
\end{equation}
where
\begin{align*}
Z_n & := \int_{t_c}^{1} \1[\beta_k(t) > N_k(t),  N_k(t_c) = m,  \hat{N}_k(t_c)= 0,  R^*_{k-1,m}(t_c) = 0] \md t \\
& \quad +   \int_{t_c}^{1} \1[\beta_k(t) < N_k(t),  N_k(t_c) = m,  \hat{N}_k(t_c)= 0,  R^*_{k-1,m}(t_c) = 0] \md t.
\end{align*}
Since $\beta_k(t), N_k(t)$ are pure jump-processes in $t$ for every fixed $n$,   $Z_n \neq 0$ iff the event $A_n := \bigcup_{t_c \leq t } \{ \beta_k(t) \neq N_k(t),  N_k(t_c) = m,  \hat{N}_k(t_c)= 0,   R^*_{k-1,m}(t_c) = 0\}$ occurs.   So,  we have that
$$ \1[Z_n > 0]  =  \1[A_n].$$
Thus,  using \eqref{e:sumzn} and Markov's inequality, we derive that for all $\epsilon > 0$, 
$$ \sum_{n \geq 1} \sP(Z_n > \epsilon) < \infty,$$
and then by Borel-Cantelli Lemma we have that $Z_n \to 0$ a.s.  as $n \to \infty$.   This yields that $\1[A_n] \to 0$ a.s.  as $n \to \infty$ and from bounded covergence theorem,  we obtain that $\sP(A_n) \to 0$ as $n \to \infty$.   So,  we have argued that \eqref{e:sumzn} implies that the the second and third terms in the inequality \eqref{e:bknkallt3} vanish asymptotically and as argued above \eqref{e:sumzn}, this proves \eqref{e:bknkallt2}.  

Setting 
\begin{align*}
P(t) &:=  \sP(\beta_k(X(n,t)) > N_k(t),  N_k(t_c) = m,  \hat{N}_k(t_c)= 0,     R^*_{k-1,m}(t_c) = 0), \\
Q(t) &:= \sP(\beta_k(X(n,t)) < N_k(t),  N_k(t_c) = m,  \hat{N}_k(t_c)= 0,     R^*_{k-1,m}(t_c) = 0),
\end{align*}
we will prove \eqref{e:sumzn},  by showing that
\begin{equation}
\label{e:Pfinite}
 \sum_{n \geq 1}  \int_{t_c}^{1} P(t)  < \infty \quad \mbox{and} \quad  \sum_{n \geq 1}  \int_{t_c}^{1} Q(t)  < \infty
\end{equation}
We now give a roadmap to rest of the proof.  In \textsc{STEP 2},  we shall break $P(t)$ into two terms $P_1(t),  P_2(t)$ and then bound each of them in \textsc{STEP 3} and \textsc{STEP 4}  respectively. In \textsc{STEP 5}, we break $Q(t)$ into three terms and bound two of them and the remaining term is bounded in \textsc{STEP 6}.  \\

\noindent \underline{\textsc{Step 2: Breaking $P(t)$ into $P_1(t)$ and $P_2(t)$.}}
Define $X'(t) = X(t) \setminus \Sigma_k(t)$ where $\Sigma_k(t)$ is the set of isolated faces.  From Lemma \ref{l:bettikX'X},
$$ \{ \beta_k(X(n,t)) > N_k(t),   \hat{N}_k(t_c)= 0,  R^*_{k-1,m}(t) = 0 \} \subset \{ \beta_k(X'(n,t)) \geq 1,   \hat{N}_k(t_c)= 0,  R^*_{k-1,m}(t) = 0  \}.$$
For $X'(n,t)$ as above and $\sigma \in \mathcal{K}_{k-1}(n)$,  define
\begin{align*}
B_1(\sigma,t) &:= \{ \mbox{$\lk_{\sigma}(X(n,t))$ has comp.  of size in $[2,N_{\sigma}(t)/2]$} \},   \\
B_2(\sigma,t) &:= \{ \lambda_2(\tlk_{\sigma}(X(n,t))) \leq \frac{k}{k+1} \},  
\end{align*}
where $N_{\sigma}(t)$ is the number of vertices in $\lk_{\sigma}(X(n,t))$ and $\tlk_{\sigma}(X(n,t))$ is the giant component in $\lk_{\sigma}(X(n,t))$.  The giant component (i.e.,  largest component) is well-defined for large $n$ (see \cite[Lemma 5.8]{SPEC} for example).  Note that $\lk_{\sigma}(X(n,t))$ is also an \ER random graph $G(N_{\sigma}(t),t)$.   Observe that  by Lemma \ref{l:lkX'conn},  we have
\begin{equation}
\label{e:tGlkX'conn}
\{ \lk_{\sigma}(X'(n,t)) \neq \tlk_{\sigma}(X(n,t)) \} =   \{ \mbox{$\lk_{\sigma}(X'(n,t))$ is not connected} \} \subset B_1(\sigma,t).
\end{equation}
Thus, using the above observations,  we derive that
\begin{align}
P(t) &  \leq   \sP\Big( \beta_k(X'(n,t)) \geq 1,   \hat{N}_k(t_c)= 0,  R^*_{k-1,m}(t) = 0 \Big) \no \\
& \leq  \sP\Big(\bigcup_{ \sigma \in X_k(n,t)}  \{ \mbox{$\lk_{\sigma}(X'(n,t))$ is not connected} \}  \Big) \no \\
& \quad  +  \sP\Big(\bigcup_{ \sigma \in X_k(n,t)}  \{ \lambda_2(\lk_{\sigma}(X'(n,t))) \leq \frac{k}{k+1} \} \cap  \{ \mbox{$\lk_{\sigma}(X'(n,t))$ is connected} \} \Big)  \no \\
& \quad \quad \quad \quad \mbox{(using Garland's theorem - Theorem \ref{thm:garland})} \no \\
& \leq  \sP\Big(\bigcup_{ \sigma \in X_k(n,t)} B_1(\sigma,t) \Big) +  \sP\Big(\bigcup_{ \sigma \in X_k(n,t)} B_2(\sigma,t) \Big), \quad \mbox{(using \eqref{e:tGlkX'conn})}   \no \\
& \leq  \sum_{\sigma  \in \mathcal{K}_{k-1}(n)}  \sP \Big( \{ \sigma \in X_k(n,t) \} \cap B_1(\sigma,t) \Big) + \sum_{\sigma  \in \mathcal{K}_{k-1}(n)}  \sP\Big( \{ \sigma \in X_k(n,t) \} \cap B_2(\sigma,t) \Big) \no 
 \\
\label{e:Ptsplit} & = n^kt^{\binom{k}{2}} \left[P_1(t) + P_2(t) \right],   \\
& \quad \quad (\mbox{ $\{ \sigma \in X_k(n,t) \}$ and $B_i(\sigma,t)$ are independent and the latter are identically distributed}) \no
\end{align}
where $P_1(t) = \sP(B_1(\sigma,t)), P_2(t) = \sP(B_2(\sigma,t))$ for some $\sigma \in \mathcal{K}_{k-1}(n)$.   Thus to prove $ \sum_{n \geq 1}  \int_{t_c}^{1} P(t)  < \infty$,  it suffices to show that $ \sum_{n \geq 1} \int_{t_c}^{1} n^kt^{\binom{k}{2}}P_i(t) < \infty$ for $i = 1,2$.  \\

\noindent \underline{\textsc{Step 3: Bounding $P_1(t)$.}} Recall $\mu(t) = (n-k)t^k, \umu(t) = \mu(t) - \mu(t)^{3/5}$ and set $\omu(t) := \mu(t) + \mu(t)^{3/5}$.
As for $P_1(t)$,  observe that
\begin{align}
P_1(t) & \leq \sP \Big( \{ \mbox{$\lk_{\sigma}(X(n,t))$ has comp.  of size in $(2,\omu(t)/2)$} \} \cap \{N_{\sigma}(t) \in (\umu(t), \omu(t)) \} \Big) \no  \\
& \quad  + \sP\Big( | N_{\sigma}(t)- \mu(t)| \geq \mu(t)^{3/5} \Big) \no \\
\label{e:P1tsplit} & =:  T_1(t)  + T_2(t)
\end{align}

We now try to bound $T_1(t)$ which is a computation about an \ER \, random graph similar to $G(nt^k,t)$. Since we are close to connectivity regime for $G(nt^k,t)$, with very high probability, $G(nt^k,t)$ is connected except for finitely many isolated nodes. Thus, we expect $P_1(t)$ to decay sufficiently fast.

Let $C_j(\sigma,t)$ be the number of components in $\lk_{\sigma}(X(n,t))$ with $j$ vertices. Set $u_j(t) = \sE \big[ C_j(\sigma,t)\1[N_{\sigma}(t) \in (\umu(t), \omu(t))] \big]$. Trivially, we have that
$$ T_1(t) \leq \sum_{j=2}^{\lfloor \omu(t)/2 \rfloor}u_j(t).$$
Since $\lk_{\sigma}(X(n,t)) \overset{d}{=} G(N_{\sigma}(t),t)$, by using the standard spanning-tree counting estimates (see for example, \cite[Section 4.1]{FK}), and conditioning on $N_{\sigma}$
we have that
$$ u_j(t) \leq \frac{\omu(t)^j}{j!}j^{j-2}t^{j-1}(1-t)^{j(\umu(t)-j)} \leq \frac{1}{tj^2}(e\omu(t)t)^j e^{-tj(\umu(t)-j)}$$
For $n$ large and $j \leq \omu(t)/2$, it holds that $\omu(t) \leq 5nt^k/4, \umu(t) - \omu(t)/2 \geq 2nt^k/5$ and hence we derive that
$$ n^k t^{\binom{k}{2}}u_j(t) \leq \frac{(n^2t^{k+1})^{k/2}}{tj^2} (\frac{5}{4}ent^{k+1}e^{-2nt^{k+1}/5})^j.$$
We can choose $\kappa$ large depending on $k$ such that for $t \geq \kappa t_c$,  we have that $2nt^{k+1}/5 \geq (\frac{k}{2}+5)\log n$ and hence, we obtain that
$$ n^k t^{\binom{k}{2}}u_j(t) \leq j^{-2}(\frac{5}{4}en^{-3/2})^j.$$
This gives that $\sum_{j = 2}^{\lfloor \omu(t)/2 \rfloor}n^k t^{\binom{k}{2}}u_j(t) \leq Cn^{-5/4}$ for $t \geq \kappa t_c$ and some constant $C$. So
\begin{equation}
\label{e:sumT1finite1}
 \sum_{n \geq 1} \int_{\kappa t_c}^{1} n^k t^{\binom{k}{2}} T_1(t) \md t  < \infty.
 \end{equation}
Now consider $t \in [t_c,\kappa t_c)$. Then, again using the above bounds on $u_j(t)$ along with the choice of $\omu(t),\umu(t)$, we obtain that for large enough $n$, 
\begin{align*}
n^k t^{\binom{k}{2}}u_j(t) &\leq \frac{1}{t_cj^2}(2k\kappa n \log n)^{k/2} (\frac{5}{4}e\kappa^{k+1} (nt_c)^{k+1}n^{-2(k/2+1)/5})^j \\
& \leq \frac{1}{t_cj^2} n^{\frac{k}{2}(1 - 2j/5) - \frac{19}{20}} (2k\kappa \log n)^{k/2} (\frac{5}{4}e\kappa^{k+1} (nt_c)^{k+1}n^{-1/20})^j.
\end{align*}
Since $\frac{k}{2}(1 - 2j/5) - \frac{19}{20} < -1$ for $k \geq 1, j \geq 3$, this suffices to derive that
$$ \sum_{n \geq 1} \int_{t_c}^{\kappa t_c}\sum_{j = 3}^{\lfloor \omu(t)/2 \rfloor}n^k t^{\binom{k}{2}}u_j(t) < \infty.$$
For $j = 2$, one can obtain a more exact bound as
$$n^k t^{\binom{k}{2}}u_2(t) \leq \frac{1}{t_cj^2}(2k\kappa n \log n)^{k/2} (\frac{5}{4}e\kappa^{k+1} (nt_c)^{k+1}e^{-t(\umu(t)-2)})^2.$$
Substituting this into the above bounds, we have that
$$\sum_{n \geq 1} \int_{t_c}^{\kappa t_c}\sum_{j = 2}^{\lfloor \omu(t)/2 \rfloor}n^k t^{\binom{k}{2}}u_j(t) < \infty$$
and hence combining with \eqref{e:sumT1finite1}, we have shown that
\begin{equation}
\label{e:sumT1finite}
 \sum_{n \geq 1} \int_{t_c}^{1} n^k t^{\binom{k}{2}} T_1(t) \md t  < \infty.
 \end{equation}

Now we bound $T_2(t)$. Since $N_{\sigma}(t)$ has binomial distribution with parameters $n-k, t^k$, using Bernstein's inequality \cite[Theorem 2.8.4]{HDP} for binomial random variables,  we can derive that 
\begin{align*}
T_2(t) &=  \sP(|N_{\sigma}(t) - \mu(t)| \leq \mu(t)^{3/5})  \leq 2 \exp \bigg(-\frac{(n-k)^{\frac{1}{5}}t^{\frac{6k}{5}}}{2(t^k(1-t^k) + (1/3)(\frac{t^{3k}}{(n-k)^2})^\frac{1}{5})}\bigg) \\
& \leq 2 \exp \bigg(-\frac{(n-k)^{\frac{1}{5}}t^{\frac{6k}{5}}}{2(t^k + (1/3)(\frac{t^{3k}}{(n-k)^2})^\frac{1}{5})}\bigg)  \\
& \leq 2 \exp \bigg(-\frac{n^{\frac{1}{5}}t_c^{\frac{6k}{5}}}{2(t_c^k + (\frac{t_c^{3k}}{n^2})^\frac{1}{5})}\bigg) \\
& \quad \quad \mbox{(as the term inside exponential is increasing in $t$)}.
\end{align*}
Now substituting the above bound,  we obtain that
\begin{align*}
& \int_{t_c}^{1} n^k t^{\binom{k}{2}} T_2(t) \md t \leq  2 \int_{t_c}^1 n^k t^{\binom{k}{2}}\exp \bigg(-\frac{n^{\frac{1}{5}}t_c^{\frac{6k}{5}}}{2(t_c^k + (\frac{t_c^{3k}}{n^2})^\frac{1}{5})}\bigg)\md t \\
& \leq 2 \int_{t_c}^1 n^k t^{\binom{k}{2}}\exp (- n^{\frac{3}{5(k+1)}} (\log n)^{\frac{3k}{5(k+1)}})\md t  \quad \mbox{(using the value of $t_c$)} \\
& \leq 2 n^k \exp (- n^{\frac{3}{5(k+1)}} (\log n)^{\frac{3k}{5(k+1)}}). 
\end{align*}
Hence,  we have that
\begin{equation}
\label{e:sumT2finite}
 \sum_{n \geq 1} \int_{t_c}^{1} n^k t^{\binom{k}{2}} T_2(t) \md t  < \infty.
 \end{equation}
So combining \eqref{e:sumT2finite} with \eqref{e:sumT1finite} and \eqref{e:P1tsplit},  we obtain that
\begin{equation}
\label{e:P1finite}
 \sum_{n \geq 1}  \int_{t_c}^{1} n^k t^{\binom{k}{2}} P_1(t)  < \infty.
\end{equation}
\\
\noindent \underline{\textsc{Step 4: Bounding $P_2(t)$.}}
We will estimate $P_2(t)$ using Theorem \ref{thm:specgapER}.   

Let $\alpha = \frac{k(k+3)}{2} - \rho$ for $\rho$ to be chosen later and recall $\mu(t) = (n-k)t^k, \umu(t) := \mu(t) - (\mu(t))^{3/5}$ be as in Lemma \ref{lem:qalpham}.  Note that $t(\umu(t)-1)$ is increasing in both $n$ and $t$ for $t \geq t_c$. Also, since $t(\umu(t)-1) \to \infty$ as $n \to \infty$, we can pick $n_0$ large enough so that Lemma \ref{lem:qalpham} applies. Also, assume that for all $n \geq n_0$ and for $C = C(\alpha)$ as in Remark \ref{rem:specgap}, it holds that
\begin{equation}
\label{e:csqrtd}
\frac{C}{\sqrt{(\umu(t)-1)t}}  \leq \frac{C}{\sqrt{(\umu(t_c)-1)t_c}}  \leq \frac{1}{k+1},  \quad t \geq t_c.
\end{equation}

Recall the notation for $q(\alpha,m)$ as before Lemma \ref{lem:qalpham} and also that $N_{\sigma}(t)$ is the number of vertices in $\lk_{\sigma}(X(n,t))$.  
\begin{align}
P_2(t) & \leq  \sP(B_2(\sigma,t) \cap \{ N_{\sigma}(t) \geq \umu(t) \} )+  \sP(N_{\sigma}(t) < \umu(t)) \no \\
& \leq \sP \big(  B_2(\sigma,t) \cap \{ t \geq q(\alpha, N_{\sigma}(t)) ,  N_{\sigma}(t) \geq \umu(t)\} \big)+  \sP(N_{\sigma}(t) < \umu(t))  \no \\
&  \quad \quad \mbox{(using $n \geq n_0$ and Lemma \ref{lem:qalpham})}  \no  \\
& \leq \sP \big(  B_2(\sigma,t) \cap \{ t \geq q(\alpha, N_{\sigma}(t)) ,  N_{\sigma}(t) \geq \umu(t)\} \big)+  \sP(| N_{\sigma}(t) - \mu(t)| \geq \mu(t)^{3/5} )  \no \\
\label{e:ZnT1T2bd} & =:  T_3(t) + T_2(t),
\end{align}
where recall that $T_2(t)$ is defined in \eqref{e:P1tsplit}.  Now,   we shall evaluate $T_3(t)$ by conditioning on $N_{\sigma}(t)$.  Observe that conditioned on $N_{\sigma}(t)$,  $\lk_{\sigma}(X(n,t))$ has the same distribution as $G(N_{\sigma},t)$.  Thus,  given $N_{\sigma}(t)$,  
$$B_2(\sigma,t)  \subset \{ |1 - \lambda_2(\tG(N_{\sigma}(t),t))| > \frac{1}{k+1} \}.$$
So, we derive that 
\begin{align*}
T_3(t) &  =   \mathbb{E} \bigg[\1[ t \geq q(\alpha, N_{\sigma}(t)),  N_{\sigma}(t) \geq \umu(t) ] \mathbb{P} \big( B_2(\sigma,t) \mid N_{\sigma}(t) \big)  \bigg] \\
& \leq \mathbb{E} \bigg[[\1[t \geq q(\alpha, N_{\sigma}(t)) , N_{\sigma}(t) \geq \umu(t) ]  \mathbb{P} \big(  |1 - \lambda_2(\tG(N_{\sigma}(t),t))| > \frac{1}{k+1} \mid N_{\sigma}(t)  \big)  \bigg]  \\
& \leq \mathbb{E} \bigg[[\1[t \geq q(\alpha, N_{\sigma}(t)) , N_{\sigma}(t) \geq \umu(t) ]   C_{\alpha}  N_{\sigma}(t)^{-2\alpha-1/2} \bigg]
\\ & \quad  \mbox{(from Remark \ref{rem:specgap} with $d = (N_{\sigma}(t)-1)t$ and \eqref{e:csqrtd}, for $n \geq n_0$)}\\
& \leq C_{\alpha}  \umu(t)^{-2\alpha-1/2}  \\
& \quad  \mbox{(since $N_{\sigma}(t) \geq \umu(t)$)}. 
\end{align*}
Without loss of generality,  we assume that for $n \geq n_0$,  $(1 - \mu(t_c)^{-2/5})^{-2\alpha-1/2} \leq 1/2$.  As $\mu(t) \geq \mu(t_c)$,  we have that for $n \geq n_0$,  $(1 - \mu(t)^{-2/5})^{-2\alpha-1/2} \leq 1/2$ for all $t \geq t_c$.  Thus,  we can derive that
\begin{align*}
& \int_{t_c}^{1} n^k t^{\binom{k}{2}} T_3(t) \md t \leq \int_{t_c}^{1} n^kt^{\binom{k}{2}} \umu(t)^{-2\alpha-1/2} \md t
\leq  \int_{t_c}^{1} n^kt^{\binom{k}{2}} (nt^k)^{-2\alpha-1/2} (1 - \mu(t)^{-2/5})^{-2\alpha-1/2}\md t\\
& \leq   \frac{1}{2}n^{k - 2\alpha-1/2} \int_{t_c}^{1} t^{k(\frac{k+1}{2} - (2\alpha + 1/2))} \md t \\
& \leq   \frac{1}{2}n^{k - 2\alpha-1/2}t_c^{k(\frac{k+1}{2} - (2\alpha + 1/2))}  \quad \mbox{($t$ has a negative exponent as $2\alpha = k(k+3) - 2\rho$ with $\rho$ small, )} \\
& =   \frac{1}{2} \bigg((\frac{k}{2}+1)\log n + \frac{k}{2}\log \log n + c\bigg)^{\frac{k}{k+1}(\frac{k+1}{2} - (2\alpha + 1/2))} n ^{\frac{k}{2} - \frac{2\alpha+1/2}{k+1}}  \quad  \mbox{(substituting the value of $t_c$)} \\
& =   \frac{1}{2} \bigg((\frac{k}{2}+1)\log n + \frac{k}{2}\log \log n + c\bigg)^{\frac{k}{k+1}(\frac{k+1}{2} - (2\alpha + 1/2))} n^{-\frac{k^2+5k+1 - 2\rho}{2(k+1)}}, 
\end{align*}
where in the last line we have used that $2\alpha = k(k+3) - 2\rho$ and we can choose a smaller $\rho$ if needed. Since the coefficient of $n$ is less than -1, we have from the above bound that
\begin{equation}
\label{e:sumT1}
\sum_{n \geq 1} \int_{t_c}^{1} n^k t^{\binom{k}{2}} T_3(t) \md t  < \infty.
\end{equation}
This along with \eqref{e:sumT2finite} and \eqref{e:ZnT1T2bd} gives 
\begin{equation}
\label{e:P2finite}
 \sum_{n \geq 1}  \int_{t_c}^{1}  n^k t^{\binom{k}{2}}  P_2(t)  < \infty.
\end{equation}

\noindent \underline{\textsc{Step 5: Bounding $Q(t)$.}}
Recall that $\Sigma_k(t)$ is the set of isolated faces in $X(n,t)$ and $X'(n,t) = X(n,t) -  \Sigma_k(t)$.   From Lemma \ref{l:xtox'} we have that $Q(t)$ can be bounded as follows.
$$Q(t) \leq \sP\Big( \beta_{k-1}(X'(t)) \neq 0,  N_k(t_c) = m,  \hat{N}_k(t_c)= 0,   R^*_{k-1,m}(t_c) = 0 \Big) .$$
Thus,  we can further bound these by using Theorem \ref{thm:garland} again as follows.
\begin{equation}
\label{e:Qtsplit}
Q(t) \leq Q_1(t) + Q_2(t) + Q_3(t),
\end{equation}
where
\begin{align*}
Q_1(t) &:= \sP \Big( \cup_{\sigma \in \mathcal{K}_{k-2}(n)} \{ \mbox{$\lk_{\sigma}(X'(n,t))$ is not connected} \} \cap \{N_k(t_c) = m,  \hat{N}_k(t_c)= 0,   R^*_{k-1,m}(t_c) = 0\} \Big),  \\
Q_2(t) &:= \sP \Big(  \cup_{\sigma \in \mathcal{K}_{k-2}(n)}  \{ | \lambda_2(\lk_{\sigma}(X'(n,t))) -    \lambda_2(\lk_{\sigma}(X(n,t))) | \geq \frac{1}{2(k+1)}\} \Big), \\
Q_3(t) &:=  \sP \Big(  \cup_{\sigma \in \mathcal{K}_{k-2}(n)}\{ | 1 -  \lambda_2(\lk_{\sigma}(X(n,t))) | \geq \frac{1}{2(k+1)}\} \Big).
\end{align*}
Using the spectral gap theorem (Theorem \ref{thm:specgapER}) as in \textsc{STEP 4} in bounding $P_2(t)$,  one can show that
\begin{equation}
\label{e:Q3finite}
 \sum_{n \geq 1}  \int_{t_c}^{1} Q_3(t)  < \infty.
\end{equation}
One can use the {\em Wielandt–Hoffman theorem} \cite{hoffweil} to bound $Q_2(t)$ as follows.  It allows us to control the spectral gap of $\lk_{\sigma}(X'(n,t))$ in terms of the spectral gap of $\lk_{\sigma}(X(n,t))$. The theorem says the following: Let $A$ and $B$ be normal matrices. Let their eigenvalues $a_i$ and $b_i$ be ordered such that $\sum_i |a_i - b_i|^2$ is minimised. Then,
\begin{equation}
\label{e:hw}
\sum_i |a_i - b_i|^2 \leq || A - B ||
\end{equation}
where $|| \cdot ||$ denotes the Frobenius matrix norm. Thus, we have
\begin{align*}
Q_2(t) &= \sP \Big(  \cup_{\sigma \in \mathcal{K}_{k-2}(n)}  \{ | \lambda_2(\lk_{\sigma}(X'(n,t))) -    \lambda_2(\lk_{\sigma}(X(n,t))) | \geq \frac{1}{2(k+1)}\} \Big)\\
& \leq  \sP \Big(  \cup_{\sigma \in \mathcal{K}_{k-2}(n)}  \{ ||\mL_{\lk_{\sigma}(X'(n,t))} - \mL_{\lk_{\sigma}(X(n,t))}|| \geq \frac{1}{2(k+1)}\} \Big) \quad  \mbox{(by \eqref{e:hw})}\\
\end{align*}
Since the entries of $\mL_{\lk_{\sigma}(X'(n,t))}$ and $\mL_{\lk_{\sigma}(X(n,t))}$ differ by the degrees in $\lk_{\sigma}$ of those vertices that form a maximal $k$-face with $\sigma$, we will overcount and bound the probability that some vertex outside $\sigma$ has low connectivity with the vertices in $\lk_{\sigma}(X(n,t))$. Note that for any vertex $v \notin \sigma$, the number of edges from $v$ to a vertex in $\lk_{\sigma}(X(n,t))$ has binomial distribution with parameters $(n-k),t^k$. Call this quantity $deg'(v)$. Then following from above, 
\begin{align*}
Q_2(t) &\leq \sP \Big(  \cup_{\sigma \in \mathcal{K}_{k-2}(n)}  \{ ||\mL_{\lk_{\sigma}(X'(n,t))} - \mL_{\lk_{\sigma}(X(n,t))}|| \geq \frac{1}{2(k+1)}\} \Big)\\
&\leq  \sP \Big(  \bigcup_{\sigma \in \mathcal{K}_{k-2}(n)} \bigcup_{v \notin \sigma} \{ deg'(v) \leq 3(k+1)\binom{m}{2} \} \Big)\\
&\leq \binom{n}{k-1} (n-k+1) \exp \Big( - \frac{(3(k+1)\binom{m}{2} - (n-k)t^k)^2}{2(n-k)t^k(1-t^k) + (1/3)(3(k+1)\binom{m}{2} - (n-k)t^k)} \Big)\\
& \quad  \mbox{(by Bernstein's inequality \cite[Theorem 2.8.4]{HDP})}\\
&\leq n^k \exp \Big( - \frac{(3(k+1)\binom{m}{2} - (n-k)t^k)^2}{2(n-k)t^k(1-t^k) + (1/3)(3(k+1)\binom{m}{2} - (n-k)t^k)} \Big)\\
&\leq n^k \exp \Big( - \frac{(3(k+1)\binom{m}{2} - nt^k)^2}{2nt^k + (1/3)(3(k+1)\binom{m}{2} - nt^k)} \Big)\\
&\leq n^k \exp \Big( - \frac{(3(k+1)\binom{m}{2} - nt_c^k)^2}{2nt_c^k + (1/3)(3(k+1)\binom{m}{2} - nt_c^k)} \Big)\\
& \sim n^k \exp ( - n^{1/(k+1)})
\end{align*}
Thus,  we have 
\begin{equation}
\label{e:Q2finite}
 \sum_{n \geq 1}  \int_{t_c}^{1} Q_2(t)  < \infty.
\end{equation}
The next and the last step of proof will be focussed on bounding $Q_1(t)$.   \\

\noindent \underline{\textsc{Step 6: Bounding $Q_1(t)$.}}

Under the event in consideration in $Q_1(t)$,  $\Sigma_k(t)$ has cardinality $m$ and since removal of faces in $\Sigma_k(t)$ disconnects $\lk_{\sigma}(X'(n,t))$, this implies that $\lk_{\sigma}(X(n,t))$ is not $m$-connected i.e., removal of $m$ edges disconnects $\lk_{\sigma}(X(n,t))$ and so does removal of $m$ vertices. This is because removal of every $k$-face can at most delete a single edge in $\lk_{\sigma}(X'(n,t))$.  Since $\lk_{\sigma}(X(n,t))$ is again distributed as $G(N_{\sigma}(t),t)$ where $N_{\sigma}(t)$ has Binomial($n,t^{k-1}$) distribution, using Markov's inequality,  we have that
\begin{equation}
\label{e:q1_ERmdiscon}
Q_1(t) \leq\sum_{\sigma \in \mathcal{K}_{k-2}(n)} \sP \Big( \mbox{$\lk_{\sigma}(X(n,t))$ is not $m$-connected}  \Big)
\end{equation}
and to bound this we argue as in \cite[Theorem 4.3]{FK}. Thus, in analogy to STEP 3, this is a computation about an \ER random graph similar to $G(nt^{k-1},t)$. Since we are well above the connectivity regime for $G(nt^{k-1},t)$, with very high probability, $G(nt^k,t)$ is $m$-connected for any $m$. Thus, we expect $Q_1(t)$ to decay sufficiently fast. We now formalize this.

Fix $\sigma \in  \mathcal{K}_{k-2}(n)$, $n(t) = nt^{k-1}$ and $N = N_{\sigma}(t)$.  Again using Bernstein's inequality as in the proof of \eqref{e:sumT2finite}, we can show that
\[ \sum_{n \geq 1} \int_{t_c}^1 n^{k-1}t^{\binom{k-1}{2}}\sP \big( |N - n(t)|  \geq n(t)^{3/5} \big) \, \md t < \infty. \]
Set $\un(t) := n(t) - n(t)^{3/5},  \on(t) := n(t) + n(t)^{3/5}$. Thus to show summability of $Q_1(t)$, it suffices to show that
\begin{equation}
\label{e:sumfinq1}
\sum_{n \geq 1} \int_{t_c}^1 n^{k-1}t^{\binom{k-1}{2}} \sP \Big( \mbox{$\lk_{\sigma}(X(n,t))$ is not $m$-connected} \cap \{N \in [\un(t),\on(t)] \} \Big) \, \md t  < \infty.
\end{equation}
  For a graph $G$ on $N$ vertices (say for $N > 2m$) to be not $m$-connected either of the following two events have to happen - (i) the graph has a vertex of degree at most $m$ or (ii) there exist disjoint subsets of vertices $S,T$ such that $|S| < m+1, m - |S| + 2 \leq |T| \leq \frac{N - |S|}{2}$ and $T$ is a component of $G - S$. In (ii), the assumption on cardinality of $T$ comes from the fact that if (i) doesn't happen and $G$ is not $m$-connected then $G-S$ has two components and so one of them must be at most $\frac{N - |S|}{2}$. Also since $T$ is a component of $G-S$, the vertices in $T$ have edges to $S$ or within $T$. Since $|S| < m$, a vertex in $T$ must have at least $m+1 - |S|$ neighbours in $T$ and this gives the lower bound on $|T|$.
  
Thus abbreviating notation, we have that
\begin{equation}
\label{e:notmconn}
 \{ \mbox{$\lk_{\sigma}(X(n,t))$ is not $m$-connected} \} \leq \{ \mbox{$\exists$ deg $ \leq m$} \}  + \{ \exists S,T  \}.
 \end{equation}

We will first bound the second term.  Now by Markov's inequality as well as using a spanning tree argument as in bounding $T_1(t)$ in STEP 3 (see also \cite[Theorem 4.3]{FK}) and conditioning on $N = N_{\sigma}(t)$,  we have that 
\begin{align*}
& \sP \Big( \exists S,T  \Big | N \Big)  \leq \sum_{j=1}^{m}\sum_{l= \max \{2,m-j+2\}}^{\frac{N-j}{2}} \binom{N}{j} \binom{N}{l}l^{l-2} t^{l-1} \binom{jl}{j} t^j (1-t)^{l(N-j-l)} \\
& \leq t^{-1}  \sum_{j=1}^{m}\sum_{l= \max \{2,m-j+2\}}^{\frac{N-j}{2}} \Big( \frac{Ne}{j} (l e) t \Big)^j l^{-2} \Big( Ne t e^{-(N-l-j)t} \Big)^l    \, \, (\mbox{using $j! \geq (j/e)^j, 1 - t \leq e^{-t}$}) \\
& \leq   t^{-1}  \sum_{j=1}^{m} \sum_{l= \max \{2,m-j+2\}}^{\frac{N-j}{2}} \Big( \frac{Ne}{j} (l e) t  \Big)^j l^{-2} \Big( Ne t e^{-(N-j)t/2} \Big)^l    \, \, (\mbox{since $j + l \leq (N+j)/2$})
\end{align*}
Thus, we can derive that for $n$ large, there is a constant $C$ (depending on $m$) such that
\[
\sP \Big( \exists S,T  \Big | N \in [\un(t),\on(t)] \Big)  \leq C t^{-1} \big( \on(t)^2t \big)^m \sum_{l= 2}^{\frac{\on(t)-1}{2}} l^{-2} \big( e \, \on(t)t \, e^{-(\un(t)-m)t/2} \big)^l
\]
Since $t \geq t_c$,  we have that $(\un(t)-m+1)t \geq \frac{n^{1/{k+1}}}{2}$ for large $n$ and so using $m,k \geq 1$ we obtain that, 
\[ \sP \Big( \exists S,T  \Big | N \in [\un(t),\on(t)] \Big) \leq  Ce^{-\frac{n^{1/{k+1}}}{4}}n^{2(m-1)} \sum_{l= 2}^{\frac{\on(t)-1}{2}} l^{-2} \big( e \, n \, e^{-\frac{n^{1/{k+1}}}{8}} \big)^l \]
Because of the exponentially fast decaying term, one can show that
\begin{equation}
\label{e:notmconnfin1}
\sum_{n \geq 1} \int_{t_c}^1 n^{k-1}t^{\binom{k-1}{2}} \sP \Big( \{ \exists S, T \} \cap \{N \in [\un(t),\on(t)] \} \Big) \, \md t  < \infty.
\end{equation}

We will now bound the probability of the first term.  Given $N$, the degree of a vertex in $\lk_{\sigma}(X(n,t))$ has Binomial($N,t$) distribution and so,
$$  \sP \Big( \mbox{$\exists$ deg $< m$} \mid N \Big) \leq \sum_{j=0}^{m-1} N t^j(1-t)^{N-j}.$$
Now using the bounds on $N$ as above, we derive that for a constant $C$ (depending on $m$),
\begin{align*}
 \sP \Big( \mbox{$\exists$ deg $< m$} | \{N \in [\un(t),\on(t)] \} \Big) & \leq \sum_{j=0}^{m-1} \on(t) t^j e^{-(\un(t)-j)t} \\
& \leq Cn e^{-\frac{n^{1/k+1}}{2}}.
\end{align*}
Thus, we can immediately obtain that
\begin{equation}
\label{e:notmconnfin2}
\sum_{n \geq 1} \int_{t_c}^1 n^{k-1}t^{\binom{k-1}{2}} \sP \Big( \{\mbox{$\exists$ deg $< m$}\} \cap \{N \in [\un(t),\on(t)] \} \Big) \, \md t  < \infty.
\end{equation}
Now from \eqref{e:sumfinq1}, \eqref{e:notmconn}, \eqref{e:notmconnfin1} and \eqref{e:notmconnfin2}, we have that
\begin{equation}
\label{e:Q1finite}
 \sum_{n \geq 1}  \int_{t_c}^{1} Q_1(t)  < \infty,
\end{equation}
and hence combined with \eqref{e:Qtsplit}, \eqref{e:Q1finite}, \eqref{e:Q2finite}, we obtain that 
\begin{equation}
\label{e:Qfinite}
 \sum_{n \geq 1}  \int_{t_c}^{1} Q(t)  < \infty.
\end{equation}

\noindent  Thus the proof of Theorem~\ref{thm:bknkapprox} is complete by combining \eqref{e:sumzn},  \eqref{e:Pfinite},  \eqref{e:Ptsplit},  \eqref{e:P1finite},  \eqref{e:P2finite} and \eqref{e:Qfinite}.
\end{proof}
\begin{remark}
\label{rem:gap_approx}
If we wanted to prove Theorem~\ref{thm:bknkapprox} for only $t >> t_c$,  then $N^*_k(t) = 0$ with high probability and this yields simplifcations to STEPS 2, 3 and 4 in the above proof and will make STEPS 5 and 6 redundant. Furthermore, though a weaker bound to that for $Q_1(t)$ is claimed in the proof of \cite[Theorem 1.3]{Fowler2019}, the proof seems to assume that $\lk_{\sigma}(X'(n,t))$ is distributed as \ER random graph (first para of Page 114 therein) which may not be true. As shown in our STEP 6, this can be done via $m$-connectivity results for \ER random graphs but does need additional work.
\end{remark}

\subsection{Proof of Main theorems - Theorems \ref{thm:main}, \ref{thm:hitting} and \ref{thm:main1}}
\label{s:proofthmmain1}

\begin{proof}(Proof of Theorem \ref{thm:main})
The proof follows from Theorems \ref{thm:isolatedprocess},  \ref{thm:bknkapprox} and Slutsky's lemma.  
\end{proof}

\begin{proof}(Proof of Theorem \ref{thm:hitting})
Observe that
\begin{align*}
\sP(T_{n,k} \neq T'_{n,k}) & \leq \sP(T_{n,k} \leq t_c) + \sP(T'_{n,k} \leq t_c) + \sP(T_{n,k} \neq T'_{n,k}, T_{n,k}, T'_{n,k} \geq t_c) \\
& = \sP(T_{n,k} \leq t_c) + \sP(T'_{n,k} \leq t_c) + \sP\Big( \bigcup_{t_c \leq t \leq 1} \{ \beta_k(X(n,t)) \neq N_k(X(n,t)) \} \Big),
\end{align*}
where the last inequality follows because if $T_{n,k} \neq T'_{n,k}$ when both are larger than $t_c$ then there exists a $t \geq t_c$ such that $\beta_k(X(n,t)) > 0 = N_k(X(n,t))$ or $\beta_k(X(n,t)) = 0 < N_k(X(n,t))$.  Now letting $n \to \infty$ and using Theorems \ref{thm:bknkapprox}, \ref{thm:main} and \ref{thm:isolatedprocess}, we obtain that
$$\lim_{n \to \infty} \sP(T_{n,k} \neq T'_{n,k})  \leq 2e^{-\mu(k,c)}.$$
Now the proof is complete by letting $c \to -\infty$ and noting that $\mu(k,c) \to \infty$. 
\end{proof}
It is possible to extend the above argument to prove the following: For $m \geq 0$, let $T_{n,k}(m) = \inf \{t: \beta_k(X(n,s)) \leq m, \, \forall s \geq t \}$ and similarly define $T'_{n,k}(m)$ with respect to $N_k(X(n,t))$.  Then we have that
$$ \lim_{n \to \infty} \sP(T_{n,k}(m) \neq T'_{n,k}(m)) = 0.$$

\begin{proof}(Proof of Theorem \ref{thm:main1})
Again,  set $t_c = t_c(1,n)$.  Using Theorem \ref{thm:zuk},  we derive that
\begin{align*}
&\sP\big(\bigcup_{t \geq t_c} \{\mbox{$\Pi_1(X(n,t))$ does not have property $(T)$}\}\big) 
\leq \sP\Big( \bigcup_{t \geq t_c} \{\mbox{$X(n,t)$ has an isolated 2-face}\} \\ 
& \quad \quad \quad \quad \bigcup \cup_{i=1}^n \{ \mbox{$\lk_i(X(t))$ is not connected} \} \bigcup \cup_{i=1}^n \{\lambda_2(\lk_i(X(t))) < \frac{1}{2}\} \Big)
\end{align*}
By an union bound, the above can be split into a sum of three terms, where the first can be bounded above by 
$ 1 - e^{-\mu(1,c)}$, which follows from Theorem \ref{thm:isolatedprocess}. In STEPS 2 and 3 of the proof of Theorem \ref{thm:bknkapprox}, we respectively show that the probability that for some $t\geq t_c(k,n)$, the link of a $(k-1)$-face in $X(n,t)$ is not connected, or has spectral gap $\lambda_2$ less than $\frac{k}{k+1}$, is asymptotically zero. Applying these bounds for $k=1$, we have that the second and third terms in above vanish.
\end{proof}

\subsection{A probabilistic lemma and proof of Lemma \ref{lem:qalpham}}
\label{s:problemmas}

Recall the quantities $R_{k-1}(t)$ and $R_{k-1}^*(t)$ from \eqref{e:rk-1mt} and \eqref{e:rk-1m*t} respectively.   Here,  we show that $\sE[R_{k-1}^*(t)]$ goes to 0.

\begin{lemma}\label{l:rk*tc}
$\lim_{n \to \infty}\mathbb{E}[R^*_{k-1}(t_c)] = 0$.
\end{lemma}

\begin{proof}
Define $\hat{R}_{k-1}(t) := R^*_{k-1}(t) - R_{k-1}(t)$.  It was shown in \cite[Section 7.2]{Fowler2019} that \\ $\lim_{n \to \infty}\mathbb{E}[R_{k-1}(t_c)] = 0$ and so to complete our proof,  we will show convergence of $\sE[\hat{R}_{k-1}(t_c)]$ by computations similar to that in Proposition \ref{prop:van_isol_faces}.  
\begin{align*}
\mathbb{E}[\hat{R}_{k-1}(t_c)] & \leq \sum_{l=0}^m \mathbb{E}\Big[\frac{1}{k!}\sum_{i_1,\dots,i_k}^{\neq} \1[\mbox{$i_1,\ldots,i_k$ do not form a $k$-clique in $G(n,t_c)$}]\\ & \times \1 [\mbox{$i_1,\ldots,i_{k}$ form a $k$-clique in $G(n,s)$ which is a subclique}\\
&\mbox{ of exactly $l$ $(k+1)$-cliques in $G(n,s)$ for some $s > t_c$}] \Big]\\
& =  \sum_{l=0}^m \int_{t_c}^1 \binom{n}{k} x^{\binom{n}{k} - 1} \binom{n-k}{l}x^{kl}(1-x^k)^{n-k-l}dx\\
& \leq \sum_{l=0}^m \int_{t_c}^1 n^{k+l} x^{\binom{n}{k} - 1} x^{kl}(1-x^k)^{n-k-l}dx\\
& \leq \sum_{l=0}^m  \int_{t_c}^1 n^{k+l} x^{\binom{n}{k} - 1} x^{kl}e^{-x^k(n-k-l)}dx\\
& \leq \sum_{l=0}^m \int_{t_c}^1 n^{k+l} x^{\binom{n}{k} - 1} x^{kl}e^{-(t_c)^k(n-k-l)}dx\\
 &\leq C\sum_{l=0}^m \int_{t_c}^1 n^{k+l} x^{n^k-1} x^{kl}e^{-((\frac{k}{2}+1)\log n)^{\frac{k}{k+1}}n^{\frac{1}{k+1}}}dx\\
& \leq (m+1)n^{k+m} e^{-((\frac{k}{2}+1)\log n)^{\frac{k}{k+1}}n^{\frac{1}{k+1}}}
\end{align*}
As the exponential term will dominate, the above will limit to 0 as $n \to \infty$.
\end{proof}
The above computation is simpler than that in Proposition \ref{prop:van_isol_faces} since we are considering $(k-1)$-faces here. 
\begin{proof}(Proof of Lemma \ref{lem:qalpham})
As $q(\alpha,m) = \frac{(\alpha+1)\log m}{m}$, which is decreasing in $m$ for $m > 1$, and $\umu(t)$ increases with $t$, it is enough to show that $t_c > q(\alpha, \umu(t_c))$. Now, 
\begin{align*}
&t_c - q(\alpha, \umu(t_c))
= t_c - \frac{(\alpha+1)\log (n t_c^k - (n t_c^k)^{\frac{3}{5}})}{n t_c^k - (n t_c^k)^{\frac{3}{5}}} 
\geq t_c - \frac{(\alpha+1)\log (n t_c^k)}{n t_c^k - (n t_c^k)^{\frac{3}{5}}}\\
=&\Big(\frac{1}{nt_c^k - (nt_c^k)^{3/5}}\Big) (nt_c^{k+1} - n^{3/5}t_c^{\frac{3k}{5}+1} - (\alpha + 1)\log(nt_c^k))
\end{align*}

Note that the denominator in the above is positive. Plugging in the value of $t_c$, we get the numerator as

\begin{align*}
&\Big((\frac{k}{2}+1)\log n + \frac{k}{2}\log \log n + c\Big) - n^{\frac{2}{5(k+1)}}\Big(\frac{k}{2}+1)\log n + \frac{k}{2}\log \log n + c\Big)^{\frac{3k+5}{5(k+1)}} - \\
 &\frac{\alpha + 1}{k+1}\log n - \frac{k(\alpha + 1)}{k+1}\log\Big((\frac{k}{2}+1)\log n + \frac{k}{2}\log \log n + c\Big)\\
= &(\frac{k}{2}+1-\frac{\alpha+1}{k+1})\log n +O(\log \log n)
\end{align*}

Plugging in $\alpha = \frac{k(k+3)}{2} - \rho$ with $\rho$ small, we get the coefficient of $\log n$ in the above is positive. This means for big enough $n$, we have the required inequality.
\end{proof}

\section*{Acknowledgements}
DY's research was funded by CPDA from the Indian Statistical Institute, DST-INSPIRE Faculty award and  SERB-MATRICS grant.  AR was partially supported by NSF grant DMS-1906414. This project originally started as part of AR's M.Math project in 2017 at Indian Statistical Institute, Bangalore. AR would also like to thank ISI Bangalore for giving the opportunity to visit and work on this problem.

\bibliography{references}
\bibliographystyle{plainnat}


\end{document}